\documentclass[reqno, 11pt]{amsart}
\usepackage{amssymb, amsmath, latexsym,enumerate}
\usepackage{graphicx,mathrsfs}
\usepackage{xcolor}
\usepackage{float, leftindex, ragged2e}

\newcounter{cnt1}
\newcounter{cnt2}
\newcounter{cnt3}
\newcommand{\blr}{\begin{list}{$($\roman{cnt1}$)$}
 {\usecounter{cnt1} \setlength{\topsep}{0pt}
 \setlength{\itemsep}{0pt}}}
\newcommand{\bla}{\begin{list}{$($\alph{cnt2}$)$}
 {\usecounter{cnt2} \setlength{\topsep}{0pt}
 \setlength{\itemsep}{0pt}}}
\newcommand{\bln}{\begin{list}{$($\arabic{cnt3}$)$}
 {\usecounter{cnt3} \setlength{\topsep}{0pt}
 \setlength{\itemsep}{0pt}}}
\newcommand{\el}{\end{list}}
\newtheorem{theorem}{Theorem}[section] 
\newtheorem{lemma}[theorem]{Lemma}

\newtheorem{example}[theorem]{Example}

\newtheorem{definition}[theorem]{Definition}
\newtheorem{proposition}[theorem]{Proposition}
\newtheorem{remark}[theorem]{Remark}
\newcommand{\Rem}{\begin{rem} \rm}
\newcommand{\bdefinition}{\begin{Def} \rm}
\newcommand{\edefinition}{\end{Def}}

\newcommand{\ba}{\begin{array}}
\newcommand{\ea}{\end{array}}

\begin{document}

\begin{center}{\Large {Applications of Fractal Sumudu Transform in Economic Models}}
		\vspace{0.05cm}
		
	  Krishna Mani Nath$^1$, Bipan Hazarika$^{1, \ast}$, Hemanta Kalita$^2$  

$^1$Department of Mathematics, Gauhati University, Guwahati, Assam, India\\
	
	$^2$Mathematics Division, School of Advanced Sciences and Languages, VIT Bhopal University, Bhopal- Indore Highway,   India\\

{Email:  kmn.math@gmail.com; bh\_rgu@yahoo.com; hemanta30kalita@gmail.com}
	
\end{center}
	\title{}
	\author{}
	\thanks{$^{\ast}$The corresponding author}
	\thanks{\today} 
	 
	\begin{abstract} 
         In this paper, we present a new fractal derivative with a nonsingular kernel and analyze its fundamental properties. The effectiveness of the proposed operator is illustrated through the study of economic models using both the Caputo fractal derivative and the new fractal derivative.
	\end{abstract}

	\maketitle
	

	\pagestyle{myheadings}
	\markboth{\rightline {\scriptsize {\textbf{\it Nath, Hazarika, Kalita} }}}
	{\leftline{ \scriptsize{ \bf \it Applications of Fractal Sumudu Transform in Economic Models}}}	
	\maketitle


\section{Introduction}
In earlier times, mathematics mainly focused on sets and functions that could be explained using the ideas of classical calculus. Sets or functions that were not smooth or regular enough were often seen as unusual or unimportant and were mostly ignored. They were treated as isolated examples rather than as part of a wider theory.

In recent years, this view has changed. Mathematicians have realized that the study of irregular or non-smooth objects is both meaningful and necessary. Such sets often describe natural shapes and patterns more accurately than traditional geometry. Fractal geometry, in particular, offers a strong and general framework for understanding and analyzing these irregular structures.

Fractal geometry, introduced by Benoit B. Mandelbrot \cite{Mandelbrot}, provided a new perspective to describe and analyze complex forms that occur in nature. It has become an essential mathematical framework for modeling and understanding systems with irregular or intricate structures. Through fractal geometry, scientists can more accurately explain and classify random or organic patterns that closely resemble natural phenomena \cite{Barnsley, Ali5, Parvate2003,Turner}.

This field has strong connections with several scientific areas, including biological sciences, astrophysics, and computer graphics. In recent decades, fractal geometry has been increasingly studied as the underlying geometry of many natural and random processes \cite{Barnsley, Ali1}. Fractals are characterized by self-similarity, in which smaller parts of the structure resemble the entire shape. In short, fractals are sets whose fractal dimension exceeds their topological dimension.


Mathematics helps economists describe and analyze complex economic problems. Using mathematical models, they can study market behavior, predict outcomes, and solve problems related to equilibrium, optimization, and dynamic changes. Tools such as calculus, optimization, and comparative statics allow economists to understand how variables such as demand, supply, and price interact. These methods help build models to predict market behavior and find strategies to maximize profit.

Price adjustment is an important concept in markets, especially in competitive markets, where demand changes continuously with price. Some firms offer refunds or compensation for short-term price changes to attract consumers and stabilize the market. A key question is whether such policies benefit sellers when buyers expect future price changes. For more information on the price adjustment strategy, see \cite{Cohen, Takayama}.

In competitive markets, many buyers and sellers interact while trying to maximize utility or profit. Economic theory, supported by mathematical models, helps explain how these interactions determine prices and quantities. Mathematics makes it easier to analyze and predict the dynamics of supply, demand, and prices.

The mathematical formulation of economic and financial systems has long relied on differential equations and mathematical models to describe the evolution of economic variables and to analyze their interrelationships. Classical differential equations provide a rigorous framework for representing dynamic processes and remain fundamental to the mathematical treatment of continuous-time economic models \cite{Nagle2011}. In mathematical economics, such models are widely employed to study the behavior of prices, demand, supply, investment, growth, and other economic quantities through deterministic and dynamic relationships \cite{Takayama}. However, many real-world economic and financial phenomena exhibit irregularities, scaling properties, memory effects, and non-smooth temporal behavior that may not be adequately captured by conventional integer-order models. In this context, fractal-based approaches provide an alternative framework for incorporating complex temporal structures into mathematical models. In particular, Golmankhaneh et al. \cite{Ali5} introduced economic models involving time fractal, demonstrating how fractal characteristics of time can be incorporated into economic dynamics and providing a broader mathematical perspective for modeling processes with nonclassical temporal behavior. Such developments motivate the investigation of generalized mathematical models in which the structure of time itself can influence the evolution of economic variables, potentially offering a more flexible representation of complex and irregular economic and financial phenomena.

The concept of competitive equilibrium is central; it occurs in a competitive market when the amount of a product that consumers want to buy matches the amount that producers are willing to sell. Mathematically, this balance can be expressed using a demand function $q_d$ and a supply function $q_s$, where each represents the quantity demanded and the quantity supplied, respectively. These two functions can be written as \begin{eqnarray}\label{i1}
    q_d(t)=d_0-d_1p(t),\;\;q_s(t)=-s_0+s_1p(t),
\end{eqnarray} where $p$ is the price of the goods, $d_0$, $s_0$, $d_1$, $s_1$ are constants affecting the quantity demanded and the quantity supplied. For the equilibrium price, we have $$q_d(t)=q_s(t)\implies p^\ast=\dfrac{d_0+s_0}{d_1+s_1}.$$ In this situation, the quantity demanded matches the quantity supplied, so neither shortage nor surplus occurs.

We consider the following simple price adjustment equation discussed in \cite{Nagle2011} \begin{align}\label{i2}
    \dfrac{dp}{dt}=\lambda(q_d-q_s),
\end{align} where $\lambda>0$. Using \eqref{i1} in \eqref{i2} yields \begin{align*}
    &\dfrac{dp}{dt}+\lambda(d_1+s_1)p(t)=\lambda(d_0+s_0)\\\implies& p(t)=\dfrac{d_0+s_0}{d_1+s_1}-\left(p(0)+\dfrac{d_0+s_0}{d_1+s_1}\right)e^{-\lambda(d_1+s_1)t},
\end{align*} where $p(0)$ is the market price at time $t=0$. On the other hand, considering the expectations of the agents, \eqref{i1} changes to \begin{align*}
    q_d(t)=d_0-d_1p(t)+d_2p'(t),\quad q_s(t)=-s_0+s_1p(t)-s_2 p'(t),
\end{align*} where $d_2,s_2\geqslant 0$. For equilibrium, we have \begin{align*}
    q_d(t)=q_s(t)\implies p'(t)-\dfrac{d_1+s_1}{d_2+s_2}p(t)=-\dfrac{d_0+s_0}{d_2+s_2}.
\end{align*}
In solving, we get \begin{align*}
    p(t)=\dfrac{d_0+s_0}{d_1+s_1}-\left[p(0)+\dfrac{d_0+s_0}{d_1+s_1}\right]e^{\dfrac{d_1+s_1}{d_2+s_2}t}.
\end{align*}
Differentiation and integration constitute fundamental mathematical tools for analyzing economic processes, formulating economic relationships, and developing mathematical models. However, the behavior and decision-making of economic agents are often influenced by previous economic conditions and historical fluctuations. In such cases, conventional integer-order derivatives and integrals may not adequately capture the effects of the past.

\section{Preliminaries to Fractal Calculus}

In this section, we recall some basic definitions related to Fractal calculus and Fractal Sumudu transforms.
\begin{definition}\cite{Parvate2003}
    The flag function for a set $F$ and a closed interval $\mathbb{I}=[x,y]$ is defined by \begin{eqnarray}\label{ms1}\Theta(F,\mathbb{I})=\begin{cases}
        1 & \text{ if }F\cap\mathbb{I}\neq\phi,\\0 & \text{otherwise.}
    \end{cases}\end{eqnarray}
\end{definition}
\begin{definition}\cite{Parvate2003}
    Let $\mathbb{P}_{[x,y]}=\{x=t_0,t_1,\ldots,t_n=y\}$ be a subdivision of $\mathbb{I}$. For a fractal set $F$ and the subdivision $\mathbb{P}_{[x,y]}$, we have \begin{eqnarray}\label{ms2}
        \sigma^\alpha[F,\mathbb{P}]=\begin{cases}
            {\displaystyle\sum_{i=0}^{n-1}\dfrac{(x_{i+1}-x_i)^\alpha}{\Gamma(\alpha+1)}\Theta(F,[x_i,x_{i+1}])} & \text{if }x<y,\\ 0  & \text{otherwise},
        \end{cases}
    \end{eqnarray}where $0<\alpha\leqslant 1$.
\end{definition}
\begin{definition}\cite{Parvate2003}
    Given $\delta>0$ and $x\leqslant y$, the coarse-grained mass of $F\cap[x,y]$ is given by \begin{eqnarray}\label{ms3}
        \gamma_\delta^\alpha(F,x,y)=\inf_{\{\mathbb{P}_{[x,y]}:|\mathbb{P}|\leqslant\delta\}}\sigma^\alpha[F,\mathbb{P}],
    \end{eqnarray} where $$|\mathbb{P}|=\max_{0\leqslant i\leqslant n-1}(x_{i+1}-x_i)$$ for a subdivision $\mathbb{P}$.
\end{definition}
\begin{definition}\cite{Parvate2003}
    The mass function $\gamma^\alpha(F,x,y)$ is given by \begin{eqnarray}\label{ms4}
        \gamma^\alpha(F,x,y)=\lim_{\delta\to 0}\gamma_\delta^\alpha(F,x,y).
    \end{eqnarray}
\end{definition}
\begin{definition}\cite{Parvate2003}
    Let $l$ be an arbitrary and fixed real number. The integral staircase function of order $\alpha$ for a set $F$ is given by \begin{eqnarray}\label{ms5}
        S_F^\alpha(x)=\begin{cases}
            \gamma^\alpha(F,l,x) & \text{if }x\geqslant l, \\
            -\gamma^\alpha(F,x,l) & \text{otherwise}.
        \end{cases}
    \end{eqnarray}
\end{definition}

\subsection{Sumudu Transform}
Consider a function $f(S_F^\alpha(t))$ defined for $t\geqslant0$ with fractal support. So, this is expressible as a polynomial or a convergent series. Let us denote the Sumudu transform of this by $\mathscr{S}(S_F^\alpha(v))$. Now, \begin{eqnarray} \label{st1}
    f(S_F^\alpha(t))= a_0+a_1S_F^\alpha(t)+a_2S_F^\alpha(t)^2+\ldots=\sum_{i=0}^{\infty}a_iS_F^\alpha(t)^i,
\end{eqnarray} where $a_0,\,a_1,\,a_2,\,\ldots$ are constants, and $a_i$ can be $0$ for some $i$.

We define  \begin{eqnarray}\label{st2}
    \mathcal{G}(S_F^\alpha(v))=a_0+a_1S_F^\alpha(v)+2!a_2S_F^\alpha(v)^2+3!a_3S_F^\alpha(v)^3+\ldots=\sum_{k=0}^{\infty}\left.\left(D_F^\alpha\right)^kf(S_F^\alpha(t))\right|_{t=0}S_F^\alpha(v)^k.
\end{eqnarray}
If $f(S_F^\alpha(t))$ is a polynomial, only finite number of terms will be involved in the expression of $\mathcal{G}(S_F^\alpha(v))$. In this case, the Sumudu transform of $f(S_F^\alpha(t))$ is \cite{Ali2}\begin{eqnarray}\label{st3}
   \mathscr{S}({S_F}^\alpha(v))=\mathcal{G}({S_F}^\alpha(v)).
\end{eqnarray}
If $f(S_F^\alpha(t))$ is a power series, there exists $M>0$ such that the series defined in (\ref{st2}) for $\mathcal{G}(S_F^\alpha(v))$ converges for $0\leqslant v<M$. Then it may also be possible to obtain a closed-form expression for $\mathcal{G}(S_F^\alpha(v))$. Let us denote the
sum of the convergent series (and the closed-form expression if it can be found) $\mathcal{G}^\ast(S_F^\alpha(v))$. Then the Sumudu transform of $f(S_F^\alpha(t))$ is $\mathcal{G}^\ast(S_F^\alpha(v))$, that is, \begin{eqnarray}\label{st4}
    \mathscr{S}(S_F^\alpha(v))=\mathcal{G}^\ast(S_F^\alpha(v)).
\end{eqnarray}
\begin{example}\cite{Ali2}
    The characteristic function with fractal support is \begin{eqnarray}\label{st5}
        \chi_F(t)=\begin{cases}
            0 & \text{when }t<0,\\\dfrac{1}{\Gamma(\alpha+1)} & \text{when }t\geqslant0.
        \end{cases}
    \end{eqnarray}Using the definition, the fractal Sumudu tranform of \normalfont{(\ref{st5})} is $$\mathscr{S}(S_F^\alpha(v))=\dfrac{1}{\Gamma(\alpha+1)},\;t\geqslant0.$$
\end{example} 
\begin{example}\cite{Ali2}
    Consider the function $g(S_F^\alpha(t))=e^{aS_F^\alpha(t)}$, where $c$ is a constant. Here \begin{eqnarray}\label{st6}e^{aS_F^\alpha(t)}=1+aS_F^\alpha(t)+\dfrac{\left(aS_F^\alpha(t)\right)^2}{2!}+\dfrac{\left(aS_F^\alpha(t)\right)^3}{3!}+\ldots.\end{eqnarray}
    Using the definition \begin{eqnarray}\label{st7}
        \mathcal{G}(S_F^\alpha(v))=1+aS_F^\alpha(v)+\left(aS_F^\alpha(v)\right)^2+\left(aS_F^\alpha(v)\right)^3+\ldots.
    \end{eqnarray} The series (\ref{st7}) converges for $\left|aS_F^\alpha(v)\right|<\left(av\right)^\alpha<1$. So, let us take $M=1/|a|^\alpha$. Then the Sumudu transform of $g(S_F^\alpha(t))$ is \begin{eqnarray}\label{st8}
        \mathscr{S}(S_F^\alpha(v))=\mathcal{G}^\ast(S_F^\alpha(v))=\dfrac{1}{1-aS_F^\alpha(v)}\;\text{ for }0\leqslant v<M=\dfrac{1}{|a|^\alpha}.
    \end{eqnarray}
\end{example}
    
\begin{definition}[The fractal Sumudu transform]\cite{Ali2}
    The fractal Sumudu transform of any function $f(S_F^\alpha(t))$ is defined by$$f(S_F^\alpha(t))\overset{\mathcal{S}_F^\alpha}{\longmapsto}\mathscr{S}\left(S_F^\alpha(v)\right)=\dfrac{1}{S_F^\alpha(v)}\int_{0}^{\infty}e^{-\dfrac{S_F^\alpha(t)}{S_F^\alpha(v)}}f(S_F^\alpha(t))d_F^\alpha t\equiv\mathcal{S}^\alpha_F[f](S_F^\alpha(t)),$$i.e.,\begin{eqnarray}\label{st9}
        \mathscr{S}(S_F^\alpha(v))=\dfrac{1}{S_F^\alpha(v)}\int_{0}^{\infty}e^{-\dfrac{S_F^\alpha(t)}{S_F^\alpha(v)}}f(S_F^\alpha(t))d_F^\alpha t,
    \end{eqnarray}if the integral exists.
\end{definition}
\begin{example}
    To find the Sumudu transform of $\sin(aS^\alpha_F(t))$ and $\cos(aS^\alpha_F(t))$, we use $$e^{iaS^\alpha_F(t)}=\cos(aS^\alpha_F(t))+i\sin(aS^\alpha_F(t)),$$ so that \begin{eqnarray}
        &&\mathscr{S}\left[\cos(aS^\alpha_F(t))+i\sin(aS^\alpha_F(t))\right](S_F^\alpha(v))=\dfrac{1}{1-iaS^\alpha_F(v)}\nonumber\\&\implies&\mathscr{S}\left[\cos(aS^\alpha_F(t))\right](S_F^\alpha(v))+i\,\mathscr{S}\left[\sin(aS^\alpha_F(t))\right](S_F^\alpha(v))=\dfrac{1+iaS^\alpha_F(v)}{1+\left(aS^\alpha_F(v)\right)^2}.\label{st91}
    \end{eqnarray} Separating real and imaginary parts of \normalfont{(\ref{st91})}, we get $$\mathscr{S}\left[\cos(aS^\alpha_F(t))\right](S_F^\alpha(v))=\dfrac{1}{1+\left(aS^\alpha_F(v)\right)^2},$$ and $$\mathscr{S}\left[\sin(aS^\alpha_F(t))\right](S_F^\alpha(v))=\dfrac{aS^\alpha_F(v)}{1+\left(aS^\alpha_F(v)\right)^2}.$$ 
\end{example} 

\begin{theorem}\cite{Ali2}
    If the Sumudu transform of $f(S_F^\alpha(t))$ is $\mathscr{S}(S_F^\alpha(v))$, then the Sumudu transform of $D_F^\alpha f(S_F^\alpha(t))$ is \begin{eqnarray}\label{st10}
        \dfrac{\mathscr{S}(S_F^\alpha(v))-\mathscr{S}(0)}{S_F^\alpha(v)}=\dfrac{\mathscr{S}(S_F^\alpha(v))-f(0)}{S_F^\alpha(v)}.
    \end{eqnarray}
\end{theorem}
\begin{theorem}\cite{Ali2}
    The Sumudu transform of $$\int_0^{S_F^\alpha(t)}f(S_F^\alpha(t))d_F^\alpha(t)$$ is given by \begin{eqnarray}\label{st11}
        \mathscr{S}\left[\int_0^{S_F^\alpha(t)}f(S_F^\alpha(t))d_F^\alpha(t)\right](S_F^\alpha(v))=S_F^\alpha(v)\mathscr{S}[f](S_F^\alpha(v)).
    \end{eqnarray}
\end{theorem}
\begin{definition}\cite{Ali3}
        The fractal Laplace transform of any function $f({S_F}^\alpha(t))\;(t\geqslant 0)$ is defined by $$f({S_F}^\alpha(t))\overset{\mathcal{L}_F^\alpha}{\longmapsto}\mathfrak{L}\left({S_F}^\alpha(v)\right)=\int_0^\infty e^{-{S_F}^\alpha(v)S_F^\alpha(t)}f({S_F}^\alpha(t))d_F^\alpha t\equiv\mathcal{L}_F^\alpha[f](S_F^\alpha(t)),$$ i.e., \begin{eqnarray}\label{rsl1}
            \mathfrak{L}(S_F^\alpha(v))=\int_0^\infty e^{-S_F^\alpha(v)S_F^\alpha(t)}f(S_F^\alpha(t))d_F^\alpha t,
        \end{eqnarray}if the integral exists. 
    \end{definition}
    Now, the following Lemma will give a relationship between the fractal Sumudu transform and the fractal Laplace transform.
    \begin{lemma}\label{rsl0}
        Let $\mathscr{S}(S_F^\alpha(v))$ and $\mathfrak{L}(S_F^\alpha(v))$ denote the fractal Sumudu transform and the fractal Laplace transform of a function $f(S_F^\alpha(t))$, respectively. Then \begin{eqnarray}\label{rsl2}
            \mathfrak{L}[f](S_F^\alpha(v))=\dfrac{1}{S_F^\alpha(v)}\mathscr{S}[f]\left(\dfrac{1}{S_F^\alpha(v)}\right).
        \end{eqnarray} Equivalently, \begin{eqnarray}\label{rsl3}
            \mathscr{S}[f](S_F^\alpha(v))=\dfrac{1}{S_F^\alpha(v)}\mathfrak{L}[f]\left(\dfrac{1}{S_F^\alpha(v)}\right).
        \end{eqnarray}
    \end{lemma}
    \begin{proof}
        We have \begin{eqnarray}\label{rsl4}
            \mathfrak{L}[f](S_F^\alpha(v))&=&\int_0^\infty e^{-S^\alpha_F(v)S_F^\alpha(t)}f(S_F^\alpha(t))d_F^\alpha t.
        \end{eqnarray}
        Take $S^\alpha_F(v)S_F^\alpha(t)=S_F^\alpha(l)$, then $S^\alpha_F(v)d_F^\alpha(t)=d_F^\alpha(l)$ so that \begin{eqnarray}\label{rsl5}
            \mathfrak{L}[f](S_F^\alpha(v))&=&\dfrac{1}{S^\alpha_F(v)}\int_0^\infty e^{-S_F^\alpha(l)}f\left(\dfrac{S_F^\alpha(l)}{S^\alpha_F(v)}\right)d_F^\alpha(l)\nonumber\\&=&\dfrac{1}{S^\alpha_F(v)}\mathscr{S}[f]\left(\dfrac{1}{S_F^\alpha(v)}\right).
        \end{eqnarray}
        The relation (\ref{rsl3}) can be shown in the similar manner.
    \end{proof}
\subsection{Sumudu transform for some non-local fractal derivatives}

    \begin{definition}[Fractal Gamma function]\cite{Ali3}
        The gamma function with fractal support is defined as \begin{eqnarray}\label{gamma1}
            \Gamma_F^\alpha(x)=\int_0^\infty e^{-S_F^\alpha(t)}S_F^\alpha(t)^{S_F^\alpha(x)-1}d_F^\alpha t,
        \end{eqnarray}where $$e^{-S_F^\alpha(t)}=F\_\lim_{n\to\infty}\left(1-\dfrac{S_F^\alpha(t)}{n}\right)^n.$$
    \end{definition}
    \begin{flushleft}
        {\bf Notation}: We shall use $C_F^\alpha[a,b]$ to denote the family of $\alpha$-order differentiable functions on $[a,b]$.
    \end{flushleft}
    \begin{definition}\cite{Ali3}
    If $f\in C_F^\alpha[a,b]$ and $\beta>0$, then the Riemann-Liouville fractal integral of order $\beta$ is defined as \begin{eqnarray}
        ^{}_a{\mathcal{I}}_x^\beta f(S_F^\alpha(x)):=\dfrac{1}{\Gamma_F^\alpha(\beta)}\int_{S_F^\alpha(a)}^{S_F^\alpha(x)}\dfrac{f(S_F^\alpha(t))}{(S_F^\alpha(x)-S_F^\alpha(t))^{\alpha-\beta}}d_F^\alpha t,\;\;S_F^\alpha(x)>S_F^\alpha(a).
    \end{eqnarray}
    \end{definition}
    \begin{theorem}
        Let $f\in C_F^\alpha[a,b]$ and $\beta>0$, then the fractal Laplace transform of ${^{}_a}{\mathcal{I}}_x^\beta f(x)$ is \cite{Ali4}\begin{eqnarray}\label{rli1}
            \mathfrak{L}\left[^{}_0{\mathcal{I}}_x^\beta f\right](S_F^\alpha(v))=\dfrac{\mathfrak{L}[f](S_F^\alpha(v))}{S_F^\alpha(v)^\beta}.
        \end{eqnarray}and the corresponding fractal Sumudu transform is\begin{eqnarray}\label{rli2}
            \mathscr{S}\left[^{}_0{\mathcal{I}}_x^\beta f\right](S_F^\alpha(v))=S_F^\alpha(v)^\beta\mathscr{S}[f](S_F^\alpha(v)).
        \end{eqnarray}
    \end{theorem}
    \begin{proof}
        (\ref{rli2}) can be proved by using the Lemma \ref{rsl0}.
    \end{proof}

    \begin{definition}\cite{Ali3}
        The fractal Mittag-Leffler function with one parameter is \begin{eqnarray}\label{ml1}
            E_F^\alpha(S_F^\alpha(t))=\sum_{j=0}^\infty\dfrac{S_F^\alpha(t)^j}{\Gamma_F^\alpha(\eta j+a)},\;\eta>0,\;a\in\mathbb{R}.
        \end{eqnarray}
    \end{definition}
    \begin{theorem}
        The fractal Sumudu transform of $E_F^\alpha(t)$ is \begin{eqnarray}\label{ml2}
            \mathscr{S}\left[E_F^\alpha\right](S_F^\alpha(v))=\sum_{j=0}^\infty\dfrac{j!S_F^\alpha(v)^j}{\Gamma_F^\alpha(\eta j+a)}.
        \end{eqnarray}
    \end{theorem}
    \begin{proof}
        It is obvious from the definition of Sumudu transform.
    \end{proof}
    \begin{lemma}\label{l1}
        For the Mittag-Leffler function $E_F^\alpha(x)$, the following results hold:\begin{enumerate}[{\normalfont (i)}]
            \item $\mathscr{S}\left[E_F^\alpha\left(-aS_F^\alpha(t)^\beta\right)\right](S_F^\alpha(v))=\dfrac{1}{1+aS_F^\alpha(v)^\beta}$,
            \item $\mathscr{S}\left[1-E_F^\alpha(-aS_F^\alpha(t)^\beta)\right](S_F^\alpha(v))=\dfrac{aS_F^\alpha(v)^\beta}{1+aS_F^\alpha(v)^\beta}$.
        \end{enumerate}
    \end{lemma}
    

    \begin{remark}\label{r1}
        For $\beta=1$, the Mittag-Leffler functions become the fractal exponential function.
    \end{remark}
        
    \begin{definition}\cite{Ali3}
        Let $f\in C_F^{\alpha n}[a,b]$, then the Riemann-Liouville fractal derivative of order $\beta\in[0,1)$ is defined as \begin{eqnarray}\label{rld1}
            {^{}_a}\mathcal{D}_b^\beta f(S_F^\alpha(x)):=\dfrac{1}{\Gamma_F^\alpha(n-\beta)}(D_F^\alpha)^n\int_{S_F^\alpha(a)}^{S_F^\alpha(x)}\dfrac{f(S_F^\alpha(t))}{(S_F^\alpha(x)-S_F^\alpha(t))^{-n+\beta+\alpha}}d_F^\alpha t,
        \end{eqnarray}where $n-\alpha\leqslant\beta<n.$
    \end{definition}
    \begin{theorem}
        Let $f\in C_F^{\alpha n}[a,b]$ and $\beta>0$, then the Laplace transform of ${^{}_0}\mathcal{D}_x^\beta f(S_F^\alpha(x))$ is \cite{Ali3} \begin{eqnarray}\label{rld2}
         \mathfrak{L}\left[{^{}_0}\mathcal{D}_x^\beta f\right](S_F^\alpha(v))=S_F^\alpha(v)^\beta\mathfrak{L}[f](S_F^\alpha(v))-\sum_{k=0}^{n-1}S_F^\alpha(v)^k\left[{^{}_0}\mathcal{D}_x^{\beta-k-1} f(S_F^\alpha(x))\right]_{x=0},
        \end{eqnarray} and the corresponding Sumudu transformation is given by \begin{eqnarray}\label{rld3}
            \mathscr{S}\left[{^{}_0}\mathcal{D}_x^\beta f\right](S_F^\alpha(v))=\dfrac{1}{S_F^\alpha(v)^\beta}\mathscr{S}[f](S_F^\alpha(v))-\sum_{k=1}^{n}\dfrac{1}{S_F^\alpha(v)^{k}}\left[{^{}_0}\mathcal{D}_x^{\beta-k}f(S_F^\alpha(x))\right]_{x=0}.
        \end{eqnarray}
     \end{theorem}
    \begin{proof}
        (\ref{rld3}) is obtained using the Lemma \ref{rsl0} in (\ref{rld2}).
    \end{proof}
    \begin{definition}\cite{Ali3}
        Let $f\in C_F^{\alpha n}[a,b]$ and $\beta>0$, then Caputo fractal derivative of order $\beta$ is given by \begin{eqnarray}\label{cd1}
            ^{C}_a\mathcal{D}_x^\beta f(S_F^\alpha(x)):=\dfrac{1}{\Gamma_F^\alpha(n-\beta)}\int_{S_F^\alpha(a)}^{S_F^\alpha(x)} \left(S_F^\alpha(x)-S_F^\alpha(t)\right)^{n\alpha-\beta-\alpha}(D_F^\alpha)^nf(S_F^\alpha(t))d_F^\alpha t,
        \end{eqnarray}where $n\alpha-\alpha\leqslant\beta<n\alpha$. 
    \end{definition}
    \begin{theorem}
        Let $f\in C_F^{\alpha n}[a,b]$ and $\beta>0$, then the Laplace tranform of $^{C}_0\mathcal{D}_x^\beta f(S_F^\alpha(x))$ is given by \cite{Ali3} \begin{eqnarray}\label{cd2}
            \mathfrak{L}\left[^{C}_0\mathcal{D}_x^\beta f\right](S_F^\alpha(v))=S_F^\alpha(v)^\beta\mathfrak{L}[f](S_F^\alpha(v))-\sum_{j=1}^{n}S_F^\alpha(v)^{\beta-j}\left[(D_F^\alpha)^{j-1}f(S_F^\alpha(x))\right]_{x=0},
        \end{eqnarray} and the corresponding  Sumudu transformation is given by \begin{eqnarray}\label{cd3}
            \mathscr{S}\left[^{C}_0\mathcal{D}_x^\beta f\right](S_F^\alpha(v))=\dfrac{1}{S_F^\alpha(v)^\beta}\mathscr{S}[f](S_F^\alpha(v))-\sum_{j=0}^{n-1}\dfrac{1}{S_F^\alpha(v)^{\beta-j}}\left[(D_F^\alpha)^{j}f(S_F^\alpha(x))\right]_{x=0}.
        \end{eqnarray}
    \end{theorem}
    \begin{proof}
        (\ref{cd3}) is obtained using the Lemma \ref{rsl0} in (\ref{cd2}).
    \end{proof}
    \begin{definition} \cite{Ali5}
        Let $f,g:[0,\infty)\to\mathbb{R}$, then the convolution of $f$ and $g$ is given by \begin{eqnarray}\label{c1}
            (f\ast g)(S_F^\alpha(t))=\int_{S_F^\alpha(0)}^{S_F^\alpha(t)}f\left(S_F^\alpha(t)-S_F^\alpha(z)\right)g\left(S_F^\alpha(z)\right)d_F^\alpha z.
        \end{eqnarray}
    \end{definition}
    \begin{proposition}
        Assume that $f,g:[0,\infty)\to\mathbb{R}$, then the Laplace transformation of the convolution $f\ast g$ is \cite{Ali5}
        \begin{eqnarray}\label{c2}
            \mathfrak{L}[f\ast g](S_F^\alpha(v))=\mathfrak{L}[f](S_F^\alpha(v))\mathfrak{L}[g](S_F^\alpha(v)).
        \end{eqnarray} and the corresponding Sumudu transformation is given by \begin{eqnarray}\label{c3}
            \mathscr{S}[f\ast g](S_F^\alpha(v))=S_F^\alpha(v)\,\mathscr{S}[f](S_F^\alpha(v))\,\mathscr{S}[g](S_F^\alpha(v)).
        \end{eqnarray}
    \end{proposition}

    \begin{proof}
        We have \begin{eqnarray}
            \mathscr{S}[f\ast g](S_F^\alpha(v))&=&\dfrac{1}{S_F^\alpha(v)}\mathfrak{L}[f\ast g]\left(\dfrac{1}{S_F^\alpha(v)}\right)\nonumber\\&=&\dfrac{1}{S_F^\alpha(v)}\mathfrak{L}[f]\left(\dfrac{1}{S_F^\alpha(v)}\right)\mathfrak{L}[g]\left(\dfrac{1}{S_F^\alpha(v)}\right)\nonumber\\&=&\dfrac{1}{S_F^\alpha(v)}\cdot S_F^\alpha(v)\mathscr{S}[f](S_F^\alpha(v))\cdot S_F^\alpha(v)\mathscr{S}[g](S_F^\alpha(v))\nonumber\\&=&S_F^\alpha(v)\mathscr{S}[f](S_F^\alpha(v))\mathscr{S}[g](S_F^\alpha(v)).\nonumber
        \end{eqnarray}
    \end{proof}

\section{Main results}
In this section, we introduce a new fractal derivative with no singular kernel. After that, we present solutions to economic models using the Caputo fractal derivative and the newly introduced wsk-fractal derivative. 

\subsection{A new fractal derivative without singular kernel}

\begin{definition}
    Let $f\in C^{\alpha n}_F[a,b]$, where $b>a$. Then a new fractal time derivative of order $\gamma\in[n-1,n)$ can be defined as \begin{eqnarray}\label{nfd1}
        {^N_a}D_F^\gamma f(S_F^\alpha(t))=\dfrac{1}{\Gamma_F^\alpha(n-\gamma)}\int_{S_F^\alpha(0)}^{S_F^\alpha(t)}\dfrac{(D^\alpha_{F,\tau})^{n}f(S_F^\alpha(\tau))}{\left(S_F^\alpha(t)-S_F^\alpha(\tau)\right)^\gamma}d_F^\alpha\tau.
    \end{eqnarray} For $n=1$, we get \begin{eqnarray}\label{nfd1.1}
        {^N_a}D_F^\gamma f(S_F^\alpha(t))=\dfrac{1}{\Gamma_F^\alpha(1-\gamma)}\int_{S_F^\alpha(0)}^{S_F^\alpha(t)}\dfrac{D^\alpha_{F,\tau}f(S_F^\alpha(\tau))}{\left(S_F^\alpha(t)-S_F^\alpha(\tau)\right)^\gamma}d_F^\alpha\tau.
    \end{eqnarray}
\end{definition}
\begin{example}[Constant function]
    For the function $f(S_F^\alpha(t))=c$, $c\in\mathbb{R}$, \normalfont{(\ref{nfd1})} gives $${^N_a}D_F^\gamma f(S_F^\alpha(t))=\dfrac{1}{\Gamma_F^\alpha(n-\gamma)}\int_{S_F^\alpha(0)}^{S_F^\alpha(t)}\dfrac{0}{\left(S_F^\alpha(t)-S_F^\alpha(\tau)\right)^\gamma}d_F^\alpha\tau=0.$$
\end{example}
\begin{example}[Power function]
    To calculate ${^N_a}D_F^\gamma (S_F^\alpha(t)^p)$, $p\in\mathbb{R}$, we start with $p>n-1$. Now \begin{eqnarray}
        ^N_aD_F^\gamma (S_F^\alpha(t)^p)&=&\dfrac{1}{\Gamma_F^\alpha(n-\gamma)}\int_{S_F^\alpha(0)}^{S_F^\alpha(t)}\dfrac{(D^\alpha_F)^n(S_F^\alpha(\tau)^p)}{\left(S_F^\alpha(t)-S_F^\alpha(\tau)\right)^{\gamma-n+1}}d_F^\alpha\tau\nonumber\\&=&\dfrac{1}{\Gamma_F^\alpha(n-\gamma)}\int_{S_F^\alpha(0)}^{S_F^\alpha(t)}\dfrac{\Gamma^\alpha_F(p+1)}{\Gamma_F^\alpha(p-n+1)}S_F^\alpha(\tau)^{p-n}\left(S_F^\alpha(t)-S_F^\alpha(\tau)\right)^{n-\gamma-1}d_F^\alpha\tau.\nonumber\\ \label{nfd2}
    \end{eqnarray}Since $0\leqslant\tau\leqslant t$, there exists $0\leqslant\lambda\leqslant1$ such that $\tau=\lambda t$ so that $S^\alpha_F(\tau)=S^\alpha_F(\lambda)S^\alpha_F(t)$ and $d_F^\alpha\tau=S_F^\alpha(\lambda)d_F^\alpha t$. Then \normalfont{(\ref{nfd2})} gives \begin{eqnarray}
        ^N_aD_F^\gamma (S_F^\alpha(t)^p)&=&\dfrac{\Gamma^\alpha_F(p+1)}{\Gamma_F^\alpha(n-\alpha)\Gamma_F^\alpha(p-n+1)}S^\alpha_F(t)^{p-\gamma}\int_0^1 S_F^\alpha(\lambda)^{p-n}\left(1-S^\alpha_F(\lambda)\right)^{n-\gamma-1}d^\alpha_F\lambda\nonumber\\&=&\dfrac{\Gamma^\alpha_F(p+1)}{\Gamma_F^\alpha(n-\gamma)\Gamma_F^\alpha(p-n+1)}S^\alpha_F(t)^{p-\gamma}\mathcal{B}^\alpha_F(p-n+1,n-\gamma),\label{nfd3}
    \end{eqnarray}where $\mathcal{B}^\alpha_F(m_1,m_2)$ is called the fractal Beta function, which is given by \cite{Ali3} \begin{eqnarray}\label{fbf1}
        \mathcal{B}_F^\alpha(m_1,m_2)=\int_0^1 S_F^\alpha(\mu)^{m_1-1}\left(1-S^\alpha_F(\mu)\right)^{m_2-1}d^\alpha_F\mu=\dfrac{\Gamma^\alpha_F(m_1)\Gamma^\alpha_F(m_2)}{\Gamma^\alpha_F(m_1+m_2)}.
    \end{eqnarray}
    Now using \normalfont{(\ref{fbf1})}, \normalfont{(\ref{nfd3})} becomes \begin{eqnarray}
        ^N_aD_F^\gamma(S_F^\alpha(t)^p)&=&\dfrac{\Gamma^\alpha_F(p+1)}{\Gamma_F^\alpha(n-\gamma)\Gamma_F^\alpha(p-n+1)}S^\alpha_F(t)^{p-\gamma}\dfrac{\Gamma^\alpha_F(p-n+1)\Gamma^\alpha_F(n-\gamma)}{\Gamma^\alpha_F(p-\gamma+1)}\nonumber\\&=&\dfrac{\Gamma^\alpha_F(p+1)}{\Gamma^\alpha_F(p-\gamma+1)}S^\alpha_F(t)^{p-\gamma}.
    \end{eqnarray}
    Again, consideration of $p\leqslant n-1$ gives $(D^\alpha_{F,\tau})^n(S_F^\alpha(\tau)^p)=0$, which leads to $${^N_a}D_F^\gamma (S_F^\alpha(t)^p)=\dfrac{1}{\Gamma_F^\alpha(n-\gamma)}\int_0^t\dfrac{0}{\left(S_F^\alpha(t)-S_F^\alpha(\tau)\right)^{\gamma-n+1}}d_F^\alpha\tau=0.$$ Thus, \begin{eqnarray}
       {^N_a}D_F^\gamma (S_F^\alpha(t)^p)=\begin{cases}
           \dfrac{\Gamma^\alpha_F(p+1)}{\Gamma^\alpha_F(p-\gamma+1)}S^\alpha_F(t)^{p-\gamma}  & \text{if }p>n-1, \\
            0 & \text{otherwise}.
       \end{cases}
    \end{eqnarray}
\end{example}
\begin{theorem}[Linearity]
    The fractal derivative given by (\ref{nfd1}) is a linear operator, that is, for any $l_1,l_2\in\mathbb{R}$, \begin{eqnarray}\label{fl1}
        {^N_a} D_F^\gamma (l_1f\left(S_F^\alpha(t))+l_2g(S_F^\alpha(t))\right)=l_1{^N_a}D_F^\gamma f\left(S_F^\alpha(t)\right)+l_2{^N_a}D_F^\gamma f\left(S_F^\alpha(t)\right).
    \end{eqnarray}
\end{theorem}
\begin{proof}
    We have \begin{eqnarray}
    {{^N}_a}D_F^\gamma (l_1f\left(S_F^\alpha(t))+l_2g(S_F^\alpha(t))\right)&=&\dfrac{1}{\Gamma_F^\alpha(n-\gamma)}\int_0^t\dfrac{(D^\alpha_{F,\tau})^{n}\left(l_1f(S^\alpha_F(\tau))+l_2g\left(S_F^\alpha(\tau)\right)\right)}{\left(S_F^\alpha(t)-S_F^\alpha(\tau)\right)^\gamma}d_F^\alpha\tau\nonumber\\&=&\dfrac{1}{\Gamma_F^\alpha(n-\gamma)}\int_0^t\dfrac{(D^\alpha_{F,\tau})^{n}\left(l_1f(S^\alpha_F(\tau))\right)}{\left(S_F^\alpha(t)-S_F^\alpha(\tau)\right)^\gamma}d_F^\alpha\tau\nonumber\\&&\hspace{1cm}+\dfrac{1}{\Gamma_F^\alpha(n-\gamma)}\int_0^t\dfrac{(D^\alpha_{F,\tau})^{n}\left(l_2f(S^\alpha_F(\tau))\right)}{\left(S_F^\alpha(t)-S_F^\alpha(\tau)\right)^\gamma}d_F^\alpha\tau\nonumber\\&=&l_1\dfrac{1}{\Gamma_F^\alpha(n-\gamma)}\int_0^t\dfrac{(D^\alpha_{F,\tau})^{n}\left(f(S^\alpha_F(\tau))\right)}{\left(S_F^\alpha(t)-S_F^\alpha(\tau)\right)^\gamma}d_F^\alpha\tau\nonumber\\&&\hspace{1cm}+l_2\dfrac{1}{\Gamma_F^\alpha(n-\gamma)}\int_0^t\dfrac{(D^\alpha_{F,\tau})^{n}\left(f(S^\alpha_F(\tau))\right)}{\left(S_F^\alpha(t)-S_F^\alpha(\tau)\right)^\gamma}d_F^\alpha\tau\nonumber\\&=&l_1{^N_a}D_F^\gamma f\left(S_F^\alpha(t)\right)+l_2{^N_a}D_F^\gamma f\left(S_F^\alpha(t)\right).\nonumber
    \end{eqnarray}
\end{proof}

\begin{definition}[wsk-fractal time derivative]
    By changing the kernel $\left(S_F^\alpha(t)-S_F^\alpha(\tau)\right)^{-\gamma}$ with the function $e^{-\dfrac{\gamma}{1-\gamma}S^\alpha_F(t)}$ and $\dfrac{1}{\Gamma_F^\alpha(1-\gamma)}$ with $\dfrac{\mathcal{N}(\gamma)}{1-\gamma}$ in \normalfont{(\ref{nfd1.1})}, we obtain the following new definition of fractal time derivative\begin{eqnarray}\label{wsk1}
    ^{\text{wsk}}\mathfrak{D}^\gamma_Ff(S_F^\alpha(t))=\dfrac{\mathcal{N}(\gamma)}{1-\gamma}\int_a^t e^{-\dfrac{\gamma(S^\alpha_F(t)-S^\alpha_F(\tau))}{1-\gamma}}D^\alpha_{F,\tau}f(S_F^\alpha(t))d^\alpha_F\tau,
    \end{eqnarray} which will be called the {\bf wsk-fractal time derivative}. Here $\mathcal{N}(\gamma)$ is a normalization function such that $\mathcal{N}(0)=\mathcal{N}(1)=1$.
\end{definition}
\begin{theorem}
    For the \normalfont{wsk-fractal time derivative}, we have \begin{eqnarray}\label{wsk2}
        \mathfrak{L}\left[^{\text{wsk}}\mathfrak{D}^\gamma_Ff\right](S_F^\alpha(v))=\dfrac{\mathcal{N}(\gamma)}{1-\gamma}\left\{\dfrac{S^\alpha_F(v)\mathfrak{L}[f](S_F^\alpha(v))-f(0)}{S^\alpha_F(v)+\dfrac{\gamma}{1-\gamma}}\right\},
    \end{eqnarray} and \begin{eqnarray}\label{wsk2.1}
        \mathscr{S}\left[^{\text{wsk}}\mathfrak{D}^\gamma_Ff\right](S_F^\alpha(v))=\dfrac{\mathcal{N}(\gamma)}{1-\gamma}\left\{\dfrac{\mathscr{S}[f](S_F^\alpha(v))-f(0)}{1+\dfrac{\gamma}{1-\gamma}S^\alpha_F(v)}\right\}.
    \end{eqnarray}
\end{theorem}
\begin{proof}
    We have \begin{eqnarray}\label{wsk3}
        \mathfrak{L}\left[^{\text{wsk}}\mathfrak{D}^\gamma_Ff\right](S_F^\alpha(v))&=&\int_0^\infty e^{-S^\alpha_F(v)S^\alpha_F(t)}\left[\dfrac{\mathcal{N}(\gamma)}{1-\gamma}\int_0^t e^{-\dfrac{\gamma(S^\alpha_F(t)-S^\alpha_F(\tau))}{1-\gamma}}D^\alpha_{F,\tau}f(\tau)d^\alpha_F\tau\right]d^\alpha_Ft\nonumber\\&=&\dfrac{\mathcal{N}(\gamma)}{1-\gamma}\int_0^\infty e^{-S^\alpha_F(v)S^\alpha_F(t)}\left[e^{-\dfrac{\gamma S^\alpha_F(t)}{1-\gamma}}\ast D^\alpha_{F,t}f(S_F^\alpha(t))\right](S_F^\alpha(v))d^\alpha_Ft\nonumber\\&=&\dfrac{\mathcal{N}(\gamma)}{1-\gamma}\mathfrak{L}\left[e^{-\dfrac{\gamma S^\alpha_F(t)}{1-\gamma}}\ast D^\alpha_{F,t}f(S_F^\alpha(t))\right](S_F^\alpha(v))\nonumber\\&=&\dfrac{\mathcal{N}(\gamma)}{1-\gamma}\mathfrak{L}\left[e^{-\dfrac{\gamma S^\alpha_F(t)}{1-\gamma}}\right](S_F^\alpha(v))\mathfrak{L}\left[D^\alpha_{F,\tau}f(S_F^\alpha(t))\right](S_F^\alpha(v)).
    \end{eqnarray} Note that \begin{eqnarray}\label{wsk4}
        \mathfrak{L}[e^{aS^\alpha_F(t)}](S_F^\alpha(v))=\dfrac{1}{S^\alpha_F(v)-a},
    \end{eqnarray} and \begin{eqnarray}\label{wsk5}
        \mathfrak{L}\left[D^\alpha_{F,t}f(S_F^\alpha(t))\right](S_F^\alpha(v))=S^\alpha_F(v)\mathfrak{L}[f](S_F^\alpha(v))-f(0).
    \end{eqnarray}
    Using (\ref{wsk4}) and (\ref{wsk5}) in (\ref{wsk3}), we get \begin{eqnarray}\label{wsk6}
        \mathfrak{L}\left[^{\text{wsk}}\mathfrak{D}^\gamma_Ff\right](S_F^\alpha(v))=\dfrac{\mathcal{N}(\gamma)}{1-\gamma}\left\{\dfrac{S^\alpha_F(v)\mathfrak{L}[f](S_F^\alpha(v))-f(0)}{S^\alpha_F(v)+\dfrac{\gamma}{1-\gamma}}\right\}.
    \end{eqnarray}
    Next, (\ref{wsk2.1}) is obtained by using the Lemma \ref{rsl0}.
\end{proof}

\subsection{Economic model formulated via the Caputo fractal derivative} \begin{enumerate}[\text{Case }1.]
    \item The price adjustment equation, formulated using the Caputo fractal derivative without the expectations of agents, is given as follows: \begin{eqnarray}\label{emc1}
    ^{C}_0\mathcal{D}_x^\beta p(S_F^\alpha(t))+\lambda(d_1+s_1)p(S_F^\alpha(t))=\lambda(d_0+s_0),\;\beta\in(0,1).
\end{eqnarray} Applying Sumudu transform, we get \begin{eqnarray}\label{emc2}
    &&\dfrac{1}{S_F^\alpha(v)^\beta}\mathscr{S}[p](S_F^\alpha(v))-\dfrac{1}{S_F^\alpha(v)^\beta}p(0)+\lambda(d_1+s_1)\mathscr{S}[p](S_F^\alpha(v))=\lambda(d_0+s_0)\nonumber\\&\implies&\left(\dfrac{1}{S_F^\alpha(v)^\beta}+\lambda(d_1+s_1)\right)\mathscr{S}[p](S_F^\alpha(v))=\dfrac{p(0)}{S_F^\alpha(v)^\beta}+\lambda(d_0+s_0)\nonumber\\&\implies&\mathscr{S}[p](S_F^\alpha(v))=\dfrac{p(0)}{1+\lambda(d_1+s_1)S_F^\alpha(v)^\beta}+\dfrac{\lambda(d_0+s_0)S_F^\alpha(v)^\beta}{1+\lambda(d_1+s_1)S_F^\alpha(v)^\beta}.
\end{eqnarray} Now, by taking inverse Sumudu transform and using Lemma \ref{l1}, we get \begin{eqnarray}\label{emc3}
    p(S_F^\alpha(t))&=&p(0)E_F^\alpha\left(-\lambda(d_1+s_1)S_F^\alpha(t)^\beta\right)+\dfrac{\lambda(d_0+s_0)}{d_1+s_1}\left(1-E_F^\alpha\left(-\lambda(d_1+s_1)S_F^\alpha(t)^\beta\right)\right).\nonumber\\
\end{eqnarray}

\item The price adjustment equation, formulated using the Caputo fractal derivative, when the expectations of agents are considered, is given as follows: \begin{eqnarray}\label{emc4}
    ^{C}_0\mathcal{D}_x^\beta p(S_F^\alpha(t))-\dfrac{d_1+s_1}{d_2+s_2}p(S_F^\alpha(t))=-\dfrac{d_0+s_0}{d_2+s_2}.
\end{eqnarray} Applying Sumudu transform, we get \begin{eqnarray}\label{emc5}
    &&\dfrac{1}{S_F^\alpha(v)^\beta}\mathscr{S}[p](S_F^\alpha(v))-\dfrac{1}{S_F^\alpha(v)^\beta}p(0)-\dfrac{d_1+s_1}{d_2+s_2}\mathscr{S}[p](S_F^\alpha(v))=-\dfrac{d_0+s_0}{d_2+s_2}\nonumber\\&\implies&\left(\dfrac{1}{S_F^\alpha(v)^\beta}-\dfrac{d_1+s_1}{d_2+s_2}\right)\mathscr{S}[p](S_F^\alpha(v))=\dfrac{p(0)}{S_F^\alpha(v)^\beta}-\dfrac{d_0+s_0}{d_2+s_2}\nonumber\\&\implies&\mathscr{S}[p](S_F^\alpha(v))=\dfrac{p(0)}{1-\dfrac{d_1+s_1}{d_2+s_2}S_F^\alpha(v)^\beta}-\dfrac{\dfrac{d_0+s_0}{d_2+s_2}S_F^\alpha(v)^\beta}{1-\dfrac{d_1+s_1}{d_2+s_2}S_F^\alpha(v)^\beta}.
\end{eqnarray}Now, by taking inverse Sumudu transform and using lemma 2.16, we get \begin{eqnarray}\label{emc6}
    p(S_F^\alpha(t))=p(0)E_F^\alpha\left(\dfrac{d_1+s_1}{d_2+s_2}S_F^\alpha(t)^\beta\right)+\dfrac{d_0+s_0}{d_1+s_1}\left(1-E_F^\alpha\left(\dfrac{d_1+s_1}{d_2+s_2}S_F^\alpha(t)^\beta\right)\right).
\end{eqnarray}
\end{enumerate}

\subsection{Economic model formulated via the wsk-fractal derivative}
\begin{enumerate}[\text{Case }1.]
    \item The price adjustment equation, formulated using the wsk-fractal derivative without the expectations of agents, is given as follows: \begin{eqnarray}\label{emw1}
    ^{\text{wsk}}\mathfrak{D}^\gamma_Fp(S_F^\alpha(t))+k(d_1+s_1)p(S_F^\alpha(t))=k(d_0+s_0).
    \end{eqnarray} Applying Sumudu transform, we get \begin{align}
        &\dfrac{\mathcal{N}(\gamma)}{1-\gamma}\left\{\dfrac{\mathscr{S}[p](S_F^\alpha(v))-p(0)}{1+\dfrac{\gamma}{1-\gamma}S^\alpha_F(v)}\right\}+k(d_1+s_1)\mathscr{S}[p](S_F^\alpha(v))=k(d_0+s_0)\nonumber\\\implies&\left[\dfrac{\mathcal{N}(\gamma)}{1-\gamma}\cdot\dfrac{1}{1+\dfrac{\gamma}{1-\gamma}S^\alpha_F(v)}+k(d_1+s_1)\right]\mathscr{S}[p](S_F^\alpha(v))\nonumber\\&\hspace{3cm}
        =\dfrac{\mathcal{N}(\gamma)}{1-\gamma}\cdot\dfrac{p(0)}{1+\dfrac{\gamma}{1-\gamma}S^\alpha_F(v)}+k(d_0+s_0)\nonumber\\\implies&\mathscr{S}[p](S_F^\alpha(v))=\dfrac{\dfrac{\mathcal{N}(\gamma)}{1-\gamma}\cdot\dfrac{p(0)}{1+\dfrac{\gamma}{1-\gamma}S^\alpha_F(v)}}{\dfrac{\mathcal{N}(\gamma)}{1-\gamma}\cdot\dfrac{1}{1+\dfrac{\gamma}{1-\gamma}S^\alpha_F(v)}+k(d_1+s_1)}+\dfrac{k(d_0+s_0)}{\dfrac{\mathcal{N}(\gamma)}{1-\gamma}\cdot\dfrac{1}{1+\dfrac{\gamma}{1-\gamma}S^\alpha_F(v)}+k(d_1+s_1)}\nonumber\\\implies&\mathscr{S}[p](S_F^\alpha(v))=\dfrac{\mathcal{N}(\gamma)p(0)}{\mathcal{N}(\gamma)+k(d_1+s_1)(1-\gamma)}\cdot\dfrac{1}{1+\dfrac{k\gamma(d_1+s_1)S_F^\alpha(v)}{\mathcal{N}(\gamma)+k(1-\gamma)(d_1+s_1)}}\nonumber\\&\hspace{4.5cm}+\dfrac{d_0+s_0}{d_1+s_1}\cdot\dfrac{\dfrac{\gamma k(d_1+s_1)}{\mathcal{N}(\gamma)+k(d_1+s_1)(1-\gamma)}S^\alpha_F(v)}{1+\dfrac{\gamma k(d_1+s_1)}{\mathcal{N}(\gamma)+k(d_1+s_1)(1-\gamma)}S^\alpha_F(v)}.\label{emw2}
    \end{align}
    Taking inverse Sumudu transform in (\ref{emw2}) and using the lemma (\ref{l1}) along with the remark (\ref{r1}), we get \begin{eqnarray}\label{emw3}
        p(S_F^\alpha(t))&=&\dfrac{\mathcal{N}(\gamma)p(0)}{\mathcal{N}(\gamma)+k(d_1+s_1)(1-\gamma)}e^{-\dfrac{k\gamma(d_1+s_1)S_F^\alpha(t)}{\mathcal{N}(\gamma)+k(1-\gamma)(d_1+s_1)}}\nonumber\\&&\hspace{2cm}+\dfrac{d_0+s_0}{d_1+s_1}\left[1-e^{-\dfrac{\gamma k(d_1+s_1)}{\mathcal{N}(\gamma)+k(d_1+s_1)(1-\gamma)}S_F^\alpha(t)}\right].
    \end{eqnarray}
    \item The price adjustment equation, formulated using the wsk-fractal derivative, when the expectations of agents are considered, is given by\begin{eqnarray}\label{emw4}
        ^{\text{wsk}}\mathfrak{D}^\gamma_Fp(S_F^\alpha(t))-\dfrac{d_1+s_1}{d_2+s_2}p(S_F^\alpha(t))=-\dfrac{d_0+s_0}{d_2+s_2}.
    \end{eqnarray} Applying Sumudu transform, we get \begin{eqnarray}
        &&\dfrac{\mathcal{N}(\gamma)}{1-\gamma}\left\{\dfrac{\mathscr{S}[p](S_F^\alpha(v))-p(0)}{1+\dfrac{\gamma}{1-\gamma}S^\alpha_F(v)}\right\}-\dfrac{d_1+s_1}{d_2+s_2}\mathscr{S}[p](S_F^\alpha(v))=-\dfrac{d_0+s_0}{d_2+s_2}\nonumber\\&\implies&\left[\dfrac{\mathcal{N}(\gamma)}{1-\gamma}\cdot\dfrac{1}{1+\dfrac{\gamma}{1-\gamma}S^\alpha_F(v)}-\dfrac{d_1+s_1}{d_2+s_2}\right]\mathscr{S}[p](S_F^\alpha(v))\nonumber\\&&\hspace{3cm}=\dfrac{\mathcal{N}(\gamma)}{1-\gamma}\cdot\dfrac{p(0)}{1+\dfrac{\gamma}{1-\gamma}S^\alpha_F(v)}-\dfrac{d_0+s_0}{d_2+s_2}\nonumber\\&\implies&\mathscr{S}[p](S_F^\alpha(v))=\dfrac{\dfrac{\mathcal{N}(\gamma)}{1-\gamma}\cdot\dfrac{p(0)}{1+\dfrac{\gamma}{1-\gamma}S^\alpha_F(v)}}{\dfrac{\mathcal{N}(\gamma)}{1-\gamma}\cdot\dfrac{1}{1+\dfrac{\gamma}{1-\gamma}S^\alpha_F(v)}-\dfrac{d_1+s_1}{d_2+s_2}}\nonumber\\&&\hspace{3cm} - \dfrac{\dfrac{d_0+s_0}{d_2+s_2}}{\dfrac{\mathcal{N}(\gamma)}{1-\gamma}\cdot\dfrac{1}{1+\dfrac{\gamma}{1-\gamma}S^\alpha_F(v)} -\dfrac{d_1+s_1}{d_2+s_2}}\nonumber\\&\implies&\mathscr{S}[p](S_F^\alpha(v))=\dfrac{\dfrac{\mathcal{N}(\gamma)p(0)(d_2+s_2)+(d_0+s_0)(1-\gamma)}{1-\dfrac{\gamma(d_1+s_1)}{\mathcal{N}(\gamma)(d_2+s_2)-(1-\gamma)(d_1+s_1)}S^\alpha_F(v)}}{\mathcal{N}(\gamma)(d_2+s_2)-(1-\gamma)(d_1+s_1)}\nonumber\\&&\hspace{2cm}-\dfrac{d_0+s_0}{d_1+s_1}\cdot\dfrac{-\dfrac{\gamma(d_1+s_1)}{\mathcal{N}(\gamma)(d_2+s_2)-(1-\gamma)(d_1+s_1)}S^\alpha_F(v)}{1-\dfrac{\gamma(d_1+s_1)}{\mathcal{N}(\gamma)(d_2+s_2)-(1-\gamma)(d_1+s_1)}S^\alpha_F(v)}.\label{emw5}
    \end{eqnarray} Taking inverse Sumudu transform in (\ref{emw5}) and using the Lemma \ref{l1} along with the Remark \ref{r1}, we get \begin{eqnarray}\label{emw6}
        p(S_F^\alpha(t))&=&\dfrac{\mathcal{N}(\gamma)p(0)(d_2+s_2)+(d_0+s_0)(1-\gamma)}{\mathcal{N}(\gamma)(d_2+s_2)-(1-\gamma)(d_1+s_1)}e^{\dfrac{\gamma(d_1+s_1)}{\mathcal{N}(\gamma)(d_2+s_2)-(1-\gamma)(d_1+s_1)}S_F^\alpha(t)}\nonumber\\&&\hspace{2cm}-\dfrac{d_0+s_0}{d_1+s_1}\left[1-e^{\dfrac{\gamma(d_1+s_1)}{\mathcal{N}(\gamma)(d_2+s_2)-(1-\gamma)(d_1+s_1)}S_F^\alpha(t)}\right].
    \end{eqnarray}
\end{enumerate}
 \section{Conclusion}
 The main findings of this study are  summarized as follows:
 \begin{enumerate}[1.]
     \item The price adjustment equation plays a crucial role in achieving market equilibrium. To incorporate the expectation-forming behavior of economic agents, fractal derivatives are employed to solve the price adjustment equation. Two distinct solutions are obtained: one corresponding to the case in which agents’ expectations are taken into account and another corresponding to the case in which expectations are not considered. Each solution is associated with a specific fractal operator.
     \item By employing non-local fractal operators that incorporate the effects of past economic changes, the relationships among the major factors influencing market equilibrium can be represented in a more comprehensive and detailed manner. These operators provide a broader framework for analyzing the memory-dependent and fractal characteristics of economic dynamics.
 \end{enumerate}

\section*{Declaration}
\noindent  \textbf{Data Availability:} No data were used or generated during the study. Therefore, data availability is not applicable. \\
\noindent
\textbf{Conflicts of Interest:} The authors declare that they have no conflicts of interest regarding the publication of this paper.\\
{\bf Author's Contributions:} All authors contributed equally to the preparation of this manuscript.\\
{\bf Funding Statement:} The authors received no financial support or funding for the research, authorship, or publication of this article.
\section*{Acknowledgments}
The first author acknowledges support of the UGC Junior Research Fellowship, funded under NTA Ref. No.: 231620032207.


\begin{thebibliography}{00}

   \bibitem{Barnsley} M.F. Barnsley, \textit{Fractals everywhere} (Academic press, 2014).

   \bibitem{Cohen} D. Cohen-Vernik, and A. Pazgal, ``Price Adjustment Policy with Partial Refunds", \textit{Journal of Retailing} 93, no. 4 (2017): 507-526. 

   \bibitem{Ali1} A. K. Golmankhaneh, {\it Fractal Calculus and its Applications: $F^\alpha$-Calculus} (World Scientific, 2022). 
   

  
   \bibitem{Ali2} A. K. Golmankhaneh, and C. Tun\c{c}, ``Sumudu transform in fractal calculus", \textit{Applied Mathematics and Computation} 350 (2019): 386-401.

   \bibitem{Ali3}  A. K. Golmankhaneh, and D. Baleanu, ``New Derivatives on the Fractal Subset of Real-Line", \textit{Entropy} 18, no. 2 (2016): 1.

   \bibitem{Ali4} A. K. Golmankhaneh, and D. Baleanu, ``Non-local Integrals and Derivatives on Fractal Sets with Applications", \textit{Open Physics}  14, no. 1 (2016): 542-548. 

   \bibitem{Ali5} A. K. Golmankhaneh, K. K. Ali, R. Yilmazer, and Mohammed K. A. Kaabar, ``Economic Models Involving Time Fractal", \textit{Journal of Mathematical Modeling in Finance} 1, no. 1 (2021): 131-146.




   \bibitem{Mandelbrot} B.B. Mandelbrot, {\it The fractal geometry of nature} (W.H. Freeman and Company, 1983).
   
   
   

   \bibitem{Nagle2011} R. K. Nagle, E. B. Saff, and A. D. Snider, \emph{Fundamentals of Differential Equations}, 8th Edition (Pearson, 2011).
   
   \bibitem{Parvate2003} A. Parvate, and A. Gangal, ``Calculus on fractal subsets of real line - I: Formulation", \textit{Fractals} 17, no. 1 (2009): 53-81. 

   \bibitem{Takayama} A. Takayama, \emph{Mathematical Economics} (Cambridge University Press, 1985).
                    
   \bibitem{Turner} M.J. Turner, {\it Modeling nature with fractals} (Leicester, 2000).


\end{thebibliography}
 \end{document}